\documentclass[reqno, 12pt]{amsart}

\usepackage{amsmath,amssymb,amsthm,amsfonts,mathrsfs}
\usepackage{fullpage}
\usepackage{relsize}
\usepackage{enumerate}
\usepackage{colonequals}
\usepackage[colorlinks=true,
linkcolor=blue,
anchorcolor=blue,
citecolor=red
]{hyperref}
\allowdisplaybreaks 
\newtheorem{theorem}{Theorem}[section]
\newtheorem{lemma}[theorem]{Lemma}
\newtheorem{corollary}[theorem]{Corollary}

\theoremstyle{definition}

\newtheorem{problem}[theorem]{Problem}

\theoremstyle{remark}
\newtheorem{remark}[theorem]{Remark}

\newcommand{\N}{\mathbb{N}}
\newcommand{\Z}{\mathbb{Z}}

\newcommand{\R}{\mathbb{R}}
\newcommand{\C}{\mathbb{C}}
\newcommand{\bG}{\mathbb{G}}
\newcommand{\bm}{\mathbf{m}}

\newcommand{\cN}{\mathcal{N}}

\newcommand{\on}{\operatorname}
\renewcommand{\i}{{\rm i}}
\renewcommand{\Re}{\on{Re}}

\renewcommand{\mod}[1]{\,(\on{mod}#1)}
\newcommand{\of}[1]{\left(#1\right)}
\newcommand{\set}[1]{\left\{#1\right\}}

\newcommand{\BEu}[1]{\underset{#1}{\mathlarger{\mathlarger{\mathbb{E}}}^{~}}\,}

\usepackage{xcolor}

\author{Biao Wang}
\address{School of Mathematics and Statistics, Yunnan University, Kunming, Yunnan 650500, China}
\email{bwang@ynu.edu.cn}

\date{\today}

\makeatletter
\@namedef{subjclassname@2020}{\textup{}2020 Mathematics Subject Classification}
\makeatother

\title[Averages of multiplicative functions]{On averages of completely multiplicative functions over co-prime integer pairs}

\begin{document}
	
\begin{abstract}
Recently, Donoso, Le, Moreira and  Sun studied the asymptotic behavior of the averages of completely multiplicative functions over the Gaussian integers. They derived Wirsing's theorem for Gaussian integers, answered a question of Frantzikinakis and Host for sum of two squares, and obtained a variant of a theorem of Bergelson and Richter on ergodic averages along the number of prime factors of integers. In this paper, we will show the analogue of these results for co-prime integer pairs. Moreover, building on Frantzikinakis and Host's results, we obtain some  convergences on the multilinear averages of multiplicative functions over primitive lattice points. 
\end{abstract}

\maketitle

\section{Introduction}

Let $n\ge1$ be an integer. Let $\Omega(n)$ be the number of prime factors of $n$ with multiplicity counted. Let $\lambda(n)=(-1)^{\Omega(n)}$ be the Liouville function, which is a bounded completely multiplicative function. It is well known that the Prime Number Theorem is equivalent to the assertion that
\begin{equation}\label{eqn_pnt_lambda}
	\lim_{N\to\infty} \frac1N\sum_{n=1}^N  \lambda(n)=0,
\end{equation}
see \cite{Landau1953, Mangoldt1897}. This leads to an intense study of the existence of the mean value
\begin{equation}\label{eqn_mf}
	M(f):= \lim_{N\to\infty} \frac1N\sum_{n=1}^N f(n)
\end{equation}
for bounded multiplicative functions $f$. Erd\H{o}s and Wintner \cite{Erdos1957} conjectured that  the limit \eqref{eqn_mf} exists, if $f$ takes the only values  $\pm1$. In 1967, by extending  both his own earlier work \cite{Wirsing1961, Wirsing1964} and the work of Delange \cite{Delange1961}, Wirsing \cite{Wirsing1967} completed the proof of the Erd\H{o}s-Wintner conjecture. Shortly after, Hal\'{a}sz \cite{Halasz1968} extended these results further  to complex-valued multiplicative functions. In 2016, by a structure theorem of multiplicative functions developed in \cite{FrantzikinakisHost2017}, Frantzikinakis and Host \cite{FrantzikinakisHost2016} extended Wirsing's theorem and an asymptotic formula of Hal\'{a}sz to multilinear averages of multiplicative functions. In \cite{FrantzikinakisHost2016}, they also proposed the following problem for homogeneous polynomials.

\begin{problem}[{\cite[Page 91]{FrantzikinakisHost2016}}]
\label{question_FH}
	    Let $f: \N \to \R$ be a real valued bounded completely multiplicative function. Let $P \in \Z[x,y]$ be a homogeneous polynomial with values on the positive integers. Does the following limit
    \[
       \lim_{N \to \infty}\frac1{N^2}\sum_{m, n=1}^N f(P(m,n))
    \]
    exist?
\end{problem}

When the polynomial $P$ is a product of linear forms, Frantzikinakis and Host already provided a positive answer to Problem~\ref{question_FH} in \cite{FrantzikinakisHost2016}.  Later, Klurman and  Mangerel \cite{KlurmanMangerel2017} established an effective asymptotic formulae in this case. And Matthiesen \cite{Matthiesen2020} studied the asymptotic formulae for related sums over dilated bounded convex subsets. When P is not a product of linear forms, Donoso, Le, Moreira and  Sun \cite{DLMS2024} solved Problem~\ref{question_FH}  for the polynomial $P(m,n)=m^2+n^2$ recently.  If $A$ is a finite set and $f:A\to\C$ is a function on $A$, we define $\BEu{x\in A}f(x)=\frac1{|A|}\sum_{x\in A}f(x)$. In this paper, we will show an equivalent statement of Problem~\ref{question_FH} in terms of the averages over co-prime integer pairs, which is stated as follows.

\begin{theorem}\label{thm_mainthm}
Let $f: \N \to \R$ be a  bounded completely multiplicative function. Let $P \in \Z[x,y]$ be a homogeneous polynomial of degree $k$ with  $P(\N^2)\subseteq \N$. Then the following limit
\begin{equation}\label{thm_mainthm_eqn1}
	\lim_{N\to\infty} \BEu{1\leq m,n\leq N}  f(P(m,n))
\end{equation}
exists if and only if the limit
\begin{equation}\label{thm_mainthm_eqn2}
	\lim_{N\to\infty} \BEu{\substack{1\leq m,n\leq N\\ \gcd(m,n)=1}} f(P(m,n))
\end{equation}
exists. If they exist, then we have
\begin{align}
	 \lim_{N\to\infty} \BEu{\substack{1\leq m,n\leq N\\ \gcd(m,n)=1}} f(P(m,n)) &= \frac{\pi^2}{6}\cdot\prod_{p}\left(1-\frac{f(p)^k}{p^2}\right)\cdot  \lim_{N\to\infty} \BEu{1\leq m,n\leq N}  f(P(m,n)), \label{thm_mainthm_eqn3}\\
	 \lim_{N\to\infty} \BEu{1\leq m,n\leq N}  f(P(m,n)) &=\frac6{\pi^2}\cdot \left(\sum_{n=1}^\infty \frac{f(n)^k}{n^2}\right)\cdot \lim_{N\to\infty} \BEu{\substack{1\leq m,n\leq N\\ \gcd(m,n)=1}} f(P(m,n)). \label{thm_mainthm_eqn4}
\end{align}
\end{theorem}

As a consequence of Theorem~\ref{thm_mainthm}, by Donoso et. al.'s result \cite[Theorem A]{DLMS2024}, the average \eqref{thm_mainthm_eqn2} also exists for $P(m,n)=m^2+n^2$. Moreover, we have
\begin{corollary}
	\label{thm_mainthm1}
	Let $f:\N\to \R$ be a bounded completely multiplicative function. Then the following limit
$$\lim_{N\to\infty} \BEu{\substack{1\leq m,n\leq N\\ \gcd(m,n)=1}} f(m^2+n^2)$$
exists and equals 
\begin{equation}
	\frac1{2-f(2)}\cdot\prod_{\substack{ p\,prime\\ p\equiv 1\mod 4}}\frac{(p+f(p))(p-1)}{(p-f(p))(p+1)}.
\end{equation}
\end{corollary}

Let $\bG=\set{m+\i n: m,n\in\Z}$ be the set of Gaussian integers and let $\bG^\ast$ be the set of non-zero Gaussian integers. Let $\cN(m+\i n)=m^2+n^2$ be the norm of Gaussian integer $m+\i n$.  Then $f(m^2+n^2)=(f\circ\cN) (m+\i n)$, which is a completely multiplicative function over $\bG^\ast$. Similar to Theorem~\ref{thm_mainthm}, by  \cite[Theorem 1.2]{DLMS2024} we have

\begin{theorem}
\label{thm_mainthm2}
	If $f: \bG^\ast \to \R$ is a real-valued bounded completely multiplicative function, then the following limit
	$$\lim_{N\to\infty} \BEu{\substack{1\leq m,n\leq N\\ \gcd(m,n)=1}} f(m+\i n)$$
	exists and equals 
$$\frac{\pi^2}{6}\cdot\prod_{p}\left(1-\frac{f(p)}{p^2}\right)\cdot \lim_{N\to\infty} \BEu{ 1\leq m,n\leq N} f(m+\i n).$$
\end{theorem}

Another way to strengthen \eqref{eqn_pnt_lambda} is a dynamical generalization  given in the recent work of Bergelson and Richter \cite{BergelsonRichter2022}. They proved that for any given uniquely ergodic topological dynamical system $(X, \mu, T)$, the following mean value 
\begin{equation}\label{eqn_BR2022thmA}
	\lim_{N\to\infty} \BEu{1\leq n\leq N}  f(T^{\Omega(n)}x)=\int_X f \,d\mu
\end{equation}
holds for any $f\in C(X)$ and $x\in X$. After Bergelson and Richter's work \cite{BergelsonRichter2022}, there are several new generalizations discovered by Loyd \cite{Loyd2022} and the related work \cite{Wang2022, WWYY2023, LWWY2024} on the ergodic theorem \eqref{eqn_BR2022thmA}. As regards the averages along $\Omega(m^2+n^2)$, Donoso et. al. \cite[Theorem D]{DLMS2024} showed that 
	\begin{equation}\label{eqn_DLMS2024thmD}
		\lim_{N\to\infty} \BEu{1\leq m,n\leq N} f(T^{\Omega(m^2+n^2)}x)=\int_Xf\,d\mu
	\end{equation}
holds for any $f\in C(X)$ and $x\in X$. In the following theorem, we provide a new variant of \eqref{eqn_BR2022thmA} related to \eqref{eqn_DLMS2024thmD}.
 
 \begin{theorem}
\label{thm_mainthm3}
	Let $(X,T, \mu)$ be a uniquely ergodic topological dynamical system. Then for any $x\in X$ and any $f\in C(X)$, we have
	\begin{equation}\label{eqn_thm_mainthm3}
		\lim_{N\to\infty} \BEu{\substack{1\leq m,n\leq N\\ \gcd(m,n)=1}} f(T^{\Omega(m^2+n^2)}x)=\int_Xf\,d\mu.
	\end{equation}
\end{theorem}

\begin{remark} Let $k\ge2$ be an integer. Two integers $a$ and $b$ are said to be relatively $k$-prime if $\gcd(a,b)$ is $k$-free. We denote it by $\gcd(a,b)_k=1$. The probability that two integers are relatively $k$-prime is $1/\zeta(2k)$, which is known as Benkoski's theorem \cite{Benkoski1976, Sittinger2010}. Then one may establish the main results for relatively $k$-prime integer pairs in this paper. For example,  similar to the proof of Theorem~\ref{thm_mainthm3}, one can prove that \eqref{eqn_thm_mainthm3} still holds if the condition $\gcd(m,n)=1$ is replaced by $\gcd(m,n)_k=1$. 
\end{remark}

For the proofs of Theorem~\ref{thm_mainthm} to Theorem~\ref{thm_mainthm3}, we develop two lemmas on the relation between summations over lattice points and summations over primitive lattice points in Section~\ref{sec_keylemmas}. Then we apply Donoso et. al.'s results to prove our results in Section~\ref{sec_proofs}. Finally, in Section~\ref{sec_multilinear}, we will give some applications of our techniques to the multilinear averages of multiplicative functions.

\section{Two lemmas}
\label{sec_keylemmas}

In this section, we consider the relation between the summations over lattice points and the summations over primitive lattice points. 

For a lattice point $\bm\in \Z^r$, we write $\bm=(m_1,\dots, m_r)$. For any $d\in \Z$, we define $d\bm=(dm_1,\dots, dm_r)$. We say $d\mid \bm$, if $\bm=d\bm'$ for some $\bm'\in \Z^r$. Define $\gcd(\bm):=\gcd(m_1,\dots, m_r)$ to be the greatest common divisor of the coordinates of $\bm$.  Clearly, $d\mid \gcd(\bm)$ if and only if $d\mid \bm$. We say $\bm$ is a \textit{primitive lattice point}, if $\gcd(\bm)=1$. The study of the counting of primitive lattice points is closely related to the generalized Gauss circle problem, see \cite{Nowak2005} for an example on the primitive lattice points inside an ellipse. 

For any positive number $a\ge 1$, with $[a]$ denote the set of natural numbers less than or equal to $a$. For instance, $[N]=\set{1,2,\dots,N}$. Then $[a]^r=\{(m_1,\dots, m_r)\in \N^r : 1\leq m_1,\dots, m_r \leq a\}$ and $|[a]^r|=\lfloor a\rfloor^r$, where $\lfloor a \rfloor$ is the largest integer less than or equal to $a$.  Let $Z_r(a)=|\set{\bm\in [a]^r: \gcd(\bm)=1}|$ be the number of primitive lattice points in $[a]^r$. Then by \cite{Nymann1972}, we have
\begin{equation}\label{eqn_coprime_r}
Z_r(N)=\frac{N^r}{\zeta(r)} + 
	\begin{cases}
	O(N^{r-1}) & \text{ if } r\ge3;\\
	O(N\log N) & \text{ if } r=2,
	\end{cases}
\end{equation}
where $\zeta(s)=\sum_{n=1}^\infty 1/{n^s}$ is the Riemann zeta function for $\Re s>1$.
We denote by $E_r(N)$ the error term in \eqref{eqn_coprime_r}.

In the following two lemmas, we will show a relation between summations over $[N]^r$ and summations over $\set{\bm\in [N]^r: \gcd(\bm)=1}$ for $N\ge1$.

\begin{lemma}
\label{lem_keylemma}
Let $r\ge 2$ be an integer. Let $b:\N^r\to\C$ be a bounded function. Let $N\ge1$. Then for any $1\leq D\leq N$, we have
	\begin{equation}\label{eqn_keylemma}
			\frac1{N^r}\sum_{\substack{\bm\in [N]^r\\ \gcd(\bm)=1}} b(\bm)= \sum_{d=1}^D\frac{\mu(d)}{d^r} \BEu{\bm\in [\frac{N}d]^r} b(d\bm)+O\of{\frac1{D^{r-1}}}+O\of{\frac{\log N}{N}},
	\end{equation}
	where $\mu(n)$ denotes the M\"obius function.
The $O$-constants in \eqref{eqn_keylemma} depend only on $b$.
\end{lemma}

\begin{proof} By $1_{\gcd(\bm)=1}=\sum_{d\mid \gcd(\bm)}\mu(d)=\sum_{d\mid \bm}\mu(d)$, we have
	\begin{align}
		\frac1{N^r}\sum_{\substack{\bm\in [N]^r\\ \gcd(\bm)=1}} b(\bm)&= \frac1{N^r}\sum_{ \bm\in [N]^r} b(\bm) \sum_{d\mid \bm}\mu(d) \nonumber\\
		&= \frac1{N^r} \sum_{1\leq d \leq N}\mu(d) \sum_{ \bm\in [\frac{N}d]^r} b(d\bm) \nonumber\\
		&= \sum_{1\leq d \leq N}\mu(d) \left(\frac1{N^r} \sum_{ \bm\in [\frac{N}d]^r} b(d\bm)\right). \label{eqn_keylemma_pf_eqn1}
	\end{align}
	
	Since $b$ is bounded, the averages of $b$ are also bounded, we have $$\BEu{ \bm\in [\frac{N}d]^r} b(d\bm) =O(1).$$ 
Hence for the inner summation in \eqref{eqn_keylemma_pf_eqn1}, 
	\begin{align*}
		\frac1{N^r} \sum_{ \bm\in [\frac{N}d]^r} b(d\bm) &= \frac1{d^r} \cdot \frac1{(N/d)^r} \sum_{ \bm\in [\frac{N}d]^r} b(d\bm)\\
		&= \frac1{d^r} \cdot \frac{\lfloor N/d \rfloor^r}{(N/d)^r} \BEu{ \bm\in [\frac{N}d]^r} b(d\bm) \\
		&= \frac1{d^r} \cdot \frac{(N/d)^r+O((N/d)^{r-1})}{(N/d)^r} \BEu{ \bm\in [\frac{N}d]^r} b(d\bm) \\
		&= \frac1{d^r}   \BEu{ \bm\in [\frac{N}d]^r} b(d\bm) + O\of{\frac1{N d^{r-1}}}.
	\end{align*}

	It follows that
	\begin{align}
		\frac1{N^r}\sum_{\substack{\bm\in [N]^r\\ \gcd(\bm)=1}} b(\bm)&=   \sum_{1\leq d \leq N}\mu(d) \left(\frac1{d^r}   \BEu{ \bm\in [\frac{N}d]^r} b(d\bm) + O\of{\frac1{N d^{r-1}}}\right) \nonumber\\
		&=   \sum_{1\leq d \leq N} \frac{\mu(d)}{d^r}   \BEu{ \bm\in [\frac{N}d]^r} b(d\bm) + O\of{\frac{\log N}{N}} \nonumber\\
		&= \sum_{1\leq d \leq D} \frac{\mu(d)}{d^r}   \BEu{ \bm\in [\frac{N}d]^r} b(d\bm) + \sum_{D< d \leq N} \frac{\mu(d)}{d^r}   \BEu{ \bm\in [\frac{N}d]^r} b(d\bm) + O\of{\frac{\log N}{N}}. \label{eqn_keylemma_pf_eqn2}
	\end{align}

For the second term in \eqref{eqn_keylemma_pf_eqn2},
\begin{equation}
\label{eqn_keylemma_pf_eqn3}
	\left|\sum_{D< d \leq N} \frac{\mu(d)}{d^r}   \BEu{ \bm\in [\frac{N}d]^r} b(d\bm) \right|\leq \sum_{d=D+1}^\infty \frac1{d^r}\left|\BEu{ \bm\in [\frac{N}d]^r} b(d\bm)\right|\ll \sum_{d=D+1}^\infty \frac1{d^{r}} \ll \frac1{D^{r-1}}.
\end{equation}

Thus, \eqref{eqn_keylemma} follows by \eqref{eqn_keylemma_pf_eqn2} and \eqref{eqn_keylemma_pf_eqn3} immediately.
\end{proof}

\begin{lemma}
\label{lem_keylemma_2}
	Let $r\ge 2$ be an integer. Let $b:\N^r\to\C$ be a bounded function. Let $N\ge1$. Then for any $1\leq D\leq N$, we have
	\begin{equation}\label{eqn_keylemma_2}
			\frac1{N^r} \sum_{\bm\in [N]^r} b(\bm)= \frac1{\zeta(r)}\sum_{d=1}^D\frac{1}{d^r} \BEu{\substack{\bm\in [\frac{N}d]^r\\ \gcd(\bm)=1}} b(d\bm)  +O\of{\frac1{D^{r-1}}}+O\of{\frac{(\log N)^2}{N}}.
	\end{equation}
The $O$-constants in \eqref{eqn_keylemma_2} depend only on $b$.
\end{lemma}

\begin{proof} We have
\begin{align*}
	\frac1{N^r} \sum_{\bm\in [N]^r} b(\bm) & =\frac1{N^r} \sum_{1\leq d\leq N} \sum_{\substack{\bm\in [N]^r\\ \gcd(\bm)=d}} b(\bm) \\
	& = \frac1{N^r} \sum_{1\leq d\leq N}\sum_{\substack{\bm\in [\frac{N}d]^r\\ \gcd(\bm)=1}} b(d\bm)\\
	&= \sum_{1\leq d\leq N} \frac{Z_r(N/d)}{N^r} \BEu{\substack{\bm\in [\frac{N}d]^r\\ \gcd(\bm)=1}} b(d\bm).
\end{align*}
Then by \eqref{eqn_coprime_r},  
\begin{align}
	\frac1{N^r} \sum_{\bm\in [N]^r} b(\bm) & = \sum_{1\leq d\leq N}\frac{(N/d)^r/\zeta(r)+E_r(N/d)}{N^r}   \BEu{\substack{\bm\in [\frac{N}d]^r\\ \gcd(\bm)=1}} b(d\bm)\nonumber\\
	&= \sum_{1\leq d\leq N} \frac1{d^r\zeta(r)} \BEu{\substack{\bm\in [\frac{N}d]^r\\ \gcd(\bm)=1}} b(d\bm) +  \sum_{1\leq d\leq N}  \frac{E_r(N/d)}{N^r} \BEu{\substack{\bm\in [\frac{N}d]^r\\ \gcd(\bm)=1}} b(d\bm). \label{eqn_keylemma_2_pf2}
\end{align}
Similar to the argument in the proof of Lemma~\ref{lem_keylemma}, the first term in \eqref{eqn_keylemma_2_pf2} is equal to
\begin{equation}\label{eqn_keylemma_2_pf3}
	\sum_{1\leq d\leq D} \frac1{d^r\zeta(r)} \BEu{\substack{\bm\in [\frac{N}d]^r\\ \gcd(\bm)=1}} b(d\bm) +O\of{\frac1{D^{r-1}}}.
\end{equation}
As regards the second term in \eqref{eqn_keylemma_2_pf2}, when $r\geq 3$,  it is
\begin{equation}\label{eqn_keylemma_2_pf4}
	\ll \sum_{1\leq d\leq N} \frac{(N/d)^{r-1}}{N^r}\ll \frac{1}{N};
\end{equation}
when $r=2$, it is
\begin{equation}\label{eqn_keylemma_2_pf5}
	\ll \sum_{1\leq d\leq N} \frac{(N/d)\log(N/d)}{N^2}\ll \frac{(\log N)^2}{N}.
\end{equation}
Thus, \eqref{eqn_keylemma_2} follows by \eqref{eqn_keylemma_2_pf2}-\eqref{eqn_keylemma_2_pf5}. 
\end{proof}

\section{Proofs of main results}
\label{sec_proofs}

In this section, we will apply Lemmas~\ref{lem_keylemma}-\ref{lem_keylemma_2} to prove Theorem~\ref{thm_mainthm} to Theorem~\ref{thm_mainthm3} respectively. Take $r=2$. By Lemmas~\ref{lem_keylemma}-\ref{lem_keylemma_2}, for any bounded function $b:\N^2\to\C$, we have
\begin{align}
	&\quad\frac1{N^2}\sum_{\substack{1\leq m,n\leq N\\ \gcd(m,n)=1}} b(m,n) = \sum_{d=1}^D\frac{\mu(d)}{d^2} \BEu{1\leq m,n\leq N/d} b(dm,dn)+O\of{\frac1D}+O\of{\frac{\log N}{N}}, \label{eqn_keylemma2}\\
&\BEu{1\leq m,n\leq N} b(m,n)= \frac{1}{\zeta(2)}\sum_{d=1}^D\frac{1}{d^2}\BEu{\substack{1\leq m,n \leq N/d\\ \gcd(m,n)=1}} b(dm,dn) +O\of{\frac{1}{D}}+O\of{\frac{(\log N)^2}{N}} \label{eqn_keylemma22}
\end{align}
for	any $N\geq1$ and any $1\leq D\leq N$.

\subsection{Proof of Theorem~\ref{thm_mainthm}} 

Take $b(m,n)=f(P(m,n))$ in \eqref{eqn_keylemma2} and  \eqref{eqn_keylemma22}, where $P(m,n)$ is a homogeneous polynomial of degree $k$ and $f$ is a bounded completely multiplicative function. Then $b(dm,dn)=f(d^kP(m,n))=f(d)^kf(P(m,n))$. By \eqref{eqn_keylemma2} and  \eqref{eqn_keylemma22}, we obtain 
\begin{align}
	&\frac1{N^2}\sum_{\substack{1\leq m,n\leq N\\ \gcd(m,n)=1}} f(P(m,n)) = \sum_{d=1}^D\frac{\mu(d)f(d)^k}{d^2} \BEu{1\leq m,n\leq N/d} f(P(m,n)) +O\of{\frac1D}+O\of{\frac{\log N}{N}}, \label{eqn_mainthm_pf1}\\
	& \BEu{1\leq m,n\leq N} f(P(m,n))= \frac{1}{\zeta(2)}\sum_{d=1}^D\frac{f(d)^k}{d^2}\BEu{\substack{1\leq m,n \leq N/d\\ \gcd(m,n)=1}} f(P(m,n)) +O\of{\frac{1}{D}}+O\of{\frac{(\log N)^2}{N}}  \label{eqn_mainthm_pf2}
\end{align}
for	any $N\geq1$ and any $1\leq D\leq N$.

Suppose the limit
$$\lim_{N\to\infty} \BEu{1\leq m,n\leq N} f(P(m,n)) $$ 
exists. Then taking $N\to \infty$ in  \eqref{eqn_mainthm_pf1}, we get
\begin{equation}\label{eqn_mainthm_pf1_1}
	\lim_{N\to\infty} \frac1{N^2}\sum_{\substack{1\leq m,n\leq N\\ \gcd(m,n)=1}} f(P(m,n)) = \sum_{d=1}^D\frac{\mu(d)f(d)^k}{d^2} \cdot \lim_{N\to\infty}\BEu{1\leq m,n\leq N} f(P(m,n)) +O\of{\frac1D}.
\end{equation}
Taking $D\to \infty$ in  \eqref{eqn_mainthm_pf1_1},  
\begin{equation}\label{eqn_mainthm_pf1_2}
	\lim_{N\to\infty} \frac1{N^2}\sum_{\substack{1\leq m,n\leq N\\ \gcd(m,n)=1}} f(P(m,n)) = \sum_{d=1}^\infty\frac{\mu(d)f(d)^k}{d^2} \cdot \lim_{N\to\infty}\BEu{1\leq m,n\leq N} f(P(m,n)).
\end{equation}

By \eqref{eqn_coprime_r}, we have 
\begin{equation}
	\label{eqn_mainthm_pf1_3}
	\sum_{\substack{1\leq m,n\leq N\\ \gcd(m,n)=1}}1=\frac{N^2}{\zeta(2)}+O(N\log N).
\end{equation}
It follows by \eqref{eqn_mainthm_pf1_2} and \eqref{eqn_mainthm_pf1_3} that the limit
$$\lim_{N\to\infty} \BEu{\substack{1\leq m,n\leq N\\ \gcd(m,n)=1}} f(P(m,n))$$
exists and equals
$$ \zeta(2)\sum_{d=1}^\infty\frac{\mu(d)f(d)^k}{d^2} \cdot \lim_{N\to\infty}\BEu{1\leq m,n\leq N} f(P(m,n)),$$
which by $\zeta(2)=\pi^2/6$ and \cite[Chapter II.1, Theorem 1.3]{Tenenbaum2015} is equal to 
$$\frac{\pi^2}{6}\cdot\prod_{p}\left(1-\frac{f(p)^k}{p^2}\right)\cdot  \lim_{N\to\infty} \BEu{1\leq m,n\leq N}  f(P(m,n)).$$
Thus, \eqref{thm_mainthm_eqn3} holds.

On the contrary, suppose the limit 
$$\lim_{N\to\infty} \BEu{\substack{1\leq m,n\leq N\\ \gcd(m,n)=1}} f(P(m,n))$$
exists. Taking $N\to \infty$ in  \eqref{eqn_mainthm_pf2}, we get
\begin{equation}\label{eqn_mainthm_pf2_1}
	\lim_{N\to\infty}\BEu{1\leq m,n\leq N} f(P(m,n))= \frac{1}{\zeta(2)} \sum_{d=1}^D\frac{f(d)^k}{d^2} \cdot \lim_{N\to\infty} \BEu{\substack{1\leq m,n\leq N\\ \gcd(m,n)=1}} f(P(m,n))+O\of{\frac{1}{D}}. 
\end{equation}
Then taking $D\to \infty$ in  \eqref{eqn_mainthm_pf2_1}, we conclude that the limit
$$\lim_{N\to\infty}\BEu{1\leq m,n\leq N} f(P(m,n))$$
exists and equals
$$\frac6{\pi^2}\cdot \left(\sum_{n=1}^\infty \frac{f(n)^k}{n^2}\right)\cdot \lim_{N\to\infty} \BEu{\substack{1\leq m,n\leq N\\ \gcd(m,n)=1}} f(P(m,n)).$$
Thus, \eqref{thm_mainthm_eqn4} holds. This completes the proof of Theorem~\ref{thm_mainthm}.\qed

\subsection{Proof of Corollary~\ref{thm_mainthm1}} 

To apply Theorem~\ref{thm_mainthm}, we first cite a theorem of Donoso et. al. on the mean value of $f(m^2+n^2)$.

\begin{theorem}[{\cite[Theorem A]{DLMS2024}}]
\label{thm_DLMS2024thmA}
	Let $f:\N\to \R$ be a bounded completely multiplicative function. Then the average
$$\lim_{N\to\infty} \BEu{1\leq m,n\leq N} f(m^2+n^2)$$
exists and equals 
\begin{equation}
\label{eqn_thm_DLMS2024thmA}
	\frac1{2-f(2)}\cdot\prod_{\substack{ p\,prime\\ p\equiv 1\mod 4}} \left(\frac{p-1}{p-f(p)}\right)^2\cdot\prod_{\substack{ p\,prime\\ p\equiv 3\mod 4}}\frac{p^2-1}{p^2-f(p)^2}.
\end{equation}
\end{theorem}

Now, take $P(m,n)=m^2+n^2$ in Theorem~\ref{thm_mainthm}. Then by Theorem~\ref{thm_DLMS2024thmA},  the limit
$$\lim_{N\to\infty} \BEu{\substack{1\leq m,n\leq N\\ \gcd(m,n)=1}} f(m^2+n^2)$$
exists and equals
$$\frac{\pi^2}{6}\cdot\prod_{p}\left(1-\frac{f(p)^2}{p^2}\right)\cdot \lim_{N\to\infty} \BEu{1\leq m,n\leq N} f(m^2+n^2),$$
which is equal to 
$$\frac1{2-f(2)}\cdot\prod_{\substack{ p\,prime\\ p\equiv 1\mod 4}}\frac{(p+f(p))(p-1)}{(p-f(p))(p+1)}$$
by plugging in \eqref{eqn_thm_DLMS2024thmA}. This completes the proof of Corollary~\ref{thm_mainthm1}.\qed

\subsection{Proof of Theorem~\ref{thm_mainthm2}} 

Take $b(m,n)=f(m+\i n)$ in \eqref{eqn_keylemma2}. Since $f$ is completely multiplicative, we have $b(dm,dn)=f(d(m+\i n))=f(d)f(m+\i n)$. By \eqref{eqn_keylemma2}, we obtain that
\begin{equation}
\label{eqn_mainthm2_pf}
	\frac1{N^2}\sum_{\substack{1\leq m,n\leq N\\ \gcd(m,n)=1}} f(m+\i n) = \sum_{d=1}^D\frac{\mu(d)f(d)}{d^2} \BEu{1\leq m,n\leq N/d} f(m+\i n) +O\of{\frac1D}+O\of{\frac{\log N}{N}}
\end{equation}
for	any $N\geq1$ and any $1\leq D\leq N$.

By \cite[Theorem 1.2.]{DLMS2024}, the limit
$$\lim_{N\to\infty} \BEu{ 1\leq m,n\leq N} f(m+\i n) $$
exists. Similar to the proof of \eqref{eqn_mainthm_pf1_2}, by taking $N\to \infty$ and then $D\to\infty$ in \eqref{eqn_mainthm2_pf},  we obtain that 
\begin{equation}\label{mainthm1.4_pf}
	\lim_{N\to\infty} \frac1{N^2}\sum_{\substack{1\leq m,n\leq N\\ \gcd(m,n)=1}} f(m+\i n) = \sum_{d=1}^\infty\frac{\mu(d)f(d)}{d^2} \cdot \BEu{1\leq m,n\leq N} f(m+\i n).
\end{equation}

By \cite[Chapter II.1, Theorem 1.3]{Tenenbaum2015}, the series $\sum_{d=1}^\infty{\mu(d)f(d)}/{d^2}$ in \eqref{mainthm1.4_pf} is equal to $\prod_{p}\left(1-f(p)/p^2\right)$. Then using \eqref{eqn_mainthm_pf1_3}  again, from \eqref{mainthm1.4_pf} we conclude that the following limit
$$\lim_{N\to\infty} \BEu{\substack{1\leq m,n\leq N\\ \gcd(m,n)=1}} f(m+\i n)$$ exists and  
$$\lim_{N\to\infty} \BEu{\substack{1\leq m,n\leq N\\ \gcd(m,n)=1}} f(m+\i n)=\frac{\pi^2}{6}\cdot\prod_{p}\left(1-\frac{f(p)}{p^2}\right)\cdot \lim_{N\to\infty} \BEu{ 1\leq m,n\leq N} f(m+\i n).$$
This completes the proof of Theorem~\ref{thm_mainthm2}.\qed

\subsection{Proof of Theorem~\ref{thm_mainthm3}} 

Take $b(m,n)=f(T^{\Omega(m^2+n^2)}x)$ in \eqref{eqn_keylemma2}. Then
$$b(dm,dn)=f(T^{\Omega(m^2+n^2)} \cdot T^{\Omega(d^2)}x).$$
By \eqref{eqn_keylemma2}, we obtain that
\begin{equation}
\label{eqn_mainthm3_pf}
	\frac1{N^2}\sum_{\substack{1\leq m,n\leq N\\ \gcd(m,n)=1}} f(T^{\Omega(m^2+n^2)}x) = \sum_{d=1}^D\frac{\mu(d)}{d^2} \BEu{1\leq m,n\leq N/d}  f(T^{\Omega(m^2+n^2)} \cdot T^{\Omega(d^2)}x) +O\of{\frac1D}+O\of{\frac{\log N}{N}}
\end{equation}
for	any $N\geq1$ and any $1\leq D\leq N$.

For each $1\leq d \leq D$, by \eqref{eqn_DLMS2024thmD} we have
$$\lim_{N\to\infty}\BEu{1\leq m,n\leq N/d}  f(T^{\Omega(m^2+n^2)} \cdot T^{\Omega(d^2)}x) =\int_X f\,d\mu. $$
Taking $N\to\infty$ and then $D\to\infty$ in \eqref{eqn_mainthm3_pf}  gives
\begin{equation}
\label{eqn_mainthm3_pf2}
	\lim_{N\to\infty} \frac1{N^2}\sum_{\substack{1\leq m,n\leq N\\ \gcd(m,n)=1}} f(T^{\Omega(m^2+n^2)}x) = \frac{6}{\pi^2} \cdot  \int_X f\,d\mu.
\end{equation}
By \eqref{eqn_mainthm_pf1_3}, we conclude that
$$\lim_{N\to\infty} \BEu{\substack{1\leq m,n\leq N\\ \gcd(m,n)=1}} f(T^{\Omega(m^2+n^2)}x)=\int_Xf\,d\mu.$$
This completes the proof of Theorem~\ref{thm_mainthm3}.\qed

\begin{remark}
	As an application of Theorem~\ref{thm_mainthm3}, one can readily show a slight generalization of Theorem~\ref{thm_mainthm3} as follows: given an intger $d\ge1$, under the same assumption as in  Theorem~\ref{thm_mainthm3}, we have
	\begin{equation}
		\lim_{N\to\infty} \BEu{\substack{1\leq m,n\leq N\\ \gcd(m,n)=d}} f(T^{\Omega(m^2+n^2)}x)=\int_Xf\,d\mu
	\end{equation}
	for any $x\in X$ and any $f\in C(X)$. One can also obtain similar generalizations of other main results for $\gcd(m,n)=d$.
\end{remark}

\section{Multilinear averages of multiplicative functions}
\label{sec_multilinear}

Let $r\ge2$. Let $f_1, \dots, f_\ell$ be complex valued multiplicative functions, and let $L_1,\dots, L_\ell: \N^r\to\N$ be linear forms given by
$$L_j(\bm)=\mathbf{k}_j\cdot \bm$$
for some $\mathbf{k}_j\in \N^r$. In \cite{FrantzikinakisHost2016}, Frantzikinakis and Host studied the asymptotic behavior of the averages
\begin{equation}\label{eqn_FH_Wirsing}
	\BEu{\bm\in [N]^r} \prod_{j=1}^\ell f_j(L_j(\bm)). 
\end{equation}
They showed Wirsing's theorem for multilinear averages \eqref{eqn_FH_Wirsing}. That is, if  $f_1, \dots, f_\ell$ are real valued multiplicative functions with modulus at most 1, then the averages  \eqref{eqn_FH_Wirsing} converge as $N\to\infty$.  For the complex valued multiplicative functions, they showed a convergence if in   \eqref{eqn_FH_Wirsing} each multiplicative function is paired up with its complex conjugate and gave an application in ergodic theory.

 In this section, we are interested in studying the convergence of the averages
 \begin{equation}\label{eqn_Wang_Wirsing}
	\BEu{\substack{\bm\in [N]^r\\ \gcd(\bm)=1}} \prod_{j=1}^\ell f_j(L_j(\bm))
\end{equation}
over primitive lattice points. Similar to the proof of Theorem~\ref{thm_mainthm}, by Lemmas~\ref{lem_keylemma}-\ref{lem_keylemma_2}, if the $f_j$'s are bounded, then  \eqref{eqn_FH_Wirsing} converges if and only if \eqref{eqn_Wang_Wirsing} converges, as $N\to\infty$. More precisely, we have the following result.

\begin{theorem}
\label{mainthm_multi}

Let $f_1, \dots, f_\ell$ be bounded complex valued multiplicative functions, and let $L_1,\dots, L_\ell: \N^r\to\N$ be linear forms. Then the following limit
\begin{equation}\label{eqn_FH_Wirsing_lim}
	\lim_{N\to\infty} \BEu{\bm\in [N]^r} \prod_{j=1}^\ell f_j(L_j(\bm))
\end{equation}
exists if and only if the limit
 \begin{equation}\label{eqn_Wang_Wirsing_lim}
	\lim_{N\to\infty}\BEu{\substack{\bm\in [N]^r\\ \gcd(\bm)=1}} \prod_{j=1}^\ell f_j(L_j(\bm))
\end{equation}
exists. If they exist, then we have
\begin{align}
	&\lim_{N\to\infty}\BEu{\substack{\bm\in [N]^r\\ \gcd(\bm)=1}} \prod_{j=1}^\ell f_j(L_j(\bm))=\zeta(r)\prod_p \left(1-\frac{\prod_{j=1}^\ell f_j(p)}{p^r}\right) \lim_{N\to\infty} \BEu{\bm\in [N]^r} \prod_{j=1}^\ell f_j(L_j(\bm)),\\
	&\lim_{N\to\infty} \BEu{\bm\in [N]^r} \prod_{j=1}^\ell f_j(L_j(\bm))= \frac1{\zeta(r)} \left(\sum_{n=1}^\infty \frac{\prod_{j=1}^\ell f_j(n)}{n^r} \right)\lim_{N\to\infty}\BEu{\substack{\bm\in [N]^r\\ \gcd(\bm)=1}} \prod_{j=1}^\ell f_j(L_j(\bm)).
\end{align}	
\end{theorem}

As an application of Theorem~\ref{mainthm_multi}, by \cite[Theorems 1.2, 1.4]{FrantzikinakisHost2016}, we have following analogue of Frantzikinakis and Host's results \cite[Theorems 1.2, 1.4]{FrantzikinakisHost2016} over primitive lattice points.

\begin{corollary}
	\begin{enumerate}
		\item Let $f_1, \dots, f_\ell$ be real valued multiplicative functions with modulus at most 1 and  $L_1,\dots, L_\ell: \N^r\to\N$ be linear forms. Then the averages
		$$\BEu{\substack{\bm\in [N]^r\\ \gcd(\bm)=1}} \prod_{j=1}^\ell f_j(L_j(\bm))$$
		converge as $N\to\infty$.

		\item Let $f_1, \dots, f_\ell$ be complex valued multiplicative functions with modulus at most 1  and  $L_j, L_j': \N^r\to\N, j=1,\dots,\ell,$ be pairwise independent linear forms.  Then the averages
		$$\BEu{\substack{\bm\in [N]^r\\ \gcd(\bm)=1}} \prod_{j=1}^\ell f_j(L_j(\bm))\cdot \overline{f_j}(L_j'(\bm))$$
		converge as $N\to\infty$.
	\end{enumerate}
\end{corollary}

Also, similar to the proof of Theorem~\ref{thm_mainthm3}, by Lemma~\ref{lem_keylemma} and Frantzikinakis and Host's result \cite[Theorem 1.5]{FrantzikinakisHost2016}, we have the following application in ergodic theory.

\begin{theorem}
\label{mainthm_multi_ergodic}

Let $(X,\mathcal{X},\mu)$ be a probability space. For $n\in \N$, let $T_n: X\to X$ be invertible measure preserving transformations that satisfy $T_1={\rm id}$ and $T_{mn}=T_m\circ T_n$ for all $m,n\in\N$. Let $L_j, L_j': \N^r\to\N, j=1,\dots,\ell,$ be pairwise independent linear forms. Then the averages
$$\BEu{\substack{\bm\in [N]^r\\ \gcd(\bm)=1}} \int F(T_{\prod_{j=1}^\ell L_j(\bm)}x) \cdot G(T_{\prod_{j=1}^\ell L_j'(\bm)}x) d\mu$$
converge for all $F, G \in L^2(\mu)$ as $N\to\infty$.
\end{theorem}

\section*{Acknowledgments}

This work is supported by the National Natural Science Foundation of China (Grant
No. 12561001). The author would like to thank the referee for a very careful review and a number of
precise comments and helpful suggestions, which improve this paper a lot.  


\end{document}